\documentclass[12pt]{article}
\usepackage{times}
\usepackage{booktabs}
\usepackage{pifont}
\usepackage{caption}
\usepackage{mathrsfs}
\usepackage{amsmath}
\usepackage{amsthm}
\usepackage{amsfonts,amsthm,amssymb,mathrsfs,bbding}
\usepackage{txfonts}
\usepackage{graphics,multicol}
\usepackage{graphicx}
\usepackage{color}
\usepackage{caption}
\usepackage{enumitem}
\usepackage{amsmath}
\usepackage{cite}
\usepackage{latexsym,bm}
\usepackage{indentfirst}
\usepackage{color}
\usepackage[backref,colorlinks=true,anchorcolor=blue,filecolor=blue,linkcolor=blue,urlcolor=blue,citecolor=blue]{hyperref}
\usepackage{extarrows}
\usepackage{cite}
\usepackage{latexsym,bm}
\usepackage{mathtools}
\evensidemargin=\oddsidemargin\topmargin=-1.5cm

\newtheorem{thm}{Theorem}[section]

\newtheorem{lem}{Lemma}[section]
\newtheorem{cor}{Corollary}[section]

\newtheorem{remark}{Remark}[section]

\newtheorem{claim}{Claim}
\newtheorem{definition}{Definition}[section]
\theoremstyle{definition}

\addtocounter{section}{0}

\def \A{\mathcal{A}}
\def \ex{\text{ex}}

\begin{document}
\title{Generalized spectral Tur\'an problems for disjoint cliques}
\author{Yi Xu, Yi-Zheng Fan\thanks{Corresponding author. Supported by National Natural Science Foundation of China (No. 12331012).
Email: yxu123@stu.ahu.edu.cn, fanyz@ahu.edu.cn}\\
\small \it Center for Pure Mathematics, School of Mathematical Sciences, \\ \small \it Anhui University, Hefei 230601, P. R. China}
\date{}
\maketitle

\begin{abstract}
The generalized Tur\'an number $\ex(n, H, F)$ denotes the maximum number of copies of $H$ in an $n$-vertex $F$-free graph.
Let $kK_{r+1}$ be the disjoint union of $k$ copies of the complete graph $K_{r+1}$.
Recently, Gerbner determined $\ex(n, K_{t},kK_{r+1})$ for all sufficiently large $n$.
In this paper, we study a spectral analogue of this problem via the $t$-clique tensor of a graph.
We prove that if an $n$-vertex $kK_{r+1}$-free graph $G$ maximizes the $t$-clique spectral radius, then for sufficiently large $n$, $G$ is the join of a complete graph $K_{k-1}$ and the $r$-partite Tur\'an graph $T_{r}(n-k+1)$.
This establishes a spectral counterpart of Gerbner's Theorem.
Moreover, in the case $t=2$, our result recovers a theorem of Ni, Wang, and Kang on the maximum spectral radius of $kK_{r+1}$-free graphs.


\end{abstract}

\section{Introduction}
A graph $G$ is said to be \emph{$H$-free} if $G$ contains no copy of $H$ as a subgraph.
A central problem of extremal graph theory asks: for a given graph $H$, what is the maximum number of edges in an $n$-vertex $H$-free graph?
This problem, together with its various extensions, is known as the \emph{Tur\'an problem}.
The {\it Tur\'an number} of $H$, denoted by $\ex(n, H)$, is defined as this maximum.
Let $T_{r}(n)$, the \emph{Tur\'an graph}, be the complete $r$-partite graph on $n$ vertices whose part sizes are as equal as possible.
A complete graph on $t$ vertices is called a $t$-clique and is denoted by $K_{t}$.
A classical result, Tur\'an's theorem \cite {Turan}, determines $\ex(n,K_{r+1})$ and shows that the unique extremal graph is $T_{r}(n)$.

Let $kK_{r+1}$ denote the disjoint union of $k$ copies of $K_{r+1}$.
For two disjoint graphs $G_{1}$ and $G_{2}$, the \emph{join} $G_{1}\vee G_{2}$ is the graph obtained from $G_1 \cup G_2$ by adding all possible edges between the vertices of $G_1$ and the vertices of $G_2$.
Simonovits \cite{Simonovits} and Moon \cite {Moon} independently proved that, for sufficiently large $n$, $K_{k-1}\vee T_{r}(n-k+1)$ is the unique extremal graph for forbidding $kK_{r+1}$; see the following theorem, where $e(G)$ denotes the number of edges of $G$.

\begin{thm}[\cite{Simonovits,Moon}]\label{edge-kK}
Let $G$ be an $n$-vertex graph with no copy of $kK_{r+1}$.
Then, for sufficiently large $n$,
$$e(G)\leq e(K_{k-1}\vee T_{r}(n-k+1)).$$
\end{thm}

A natural generalization of Tur\'an type problems is to count copies of a fixed subgraph rather than just an edge.
Given a graph $H$ and a family $\mathcal{F}$ of graphs, let $\ex(n, H, \mathcal{F})$ denote the maximum number of copies of $H$ in an $n$-vertex graph that contains no member of $\mathcal{F}$ as a subgraph.
This quantity is called the \emph{generalized Tur\'an number}.
When $\mathcal{F} = \{F\}$, we simply write $\ex(n, H, F)$.
In particular, when $H=K_2$, then $\ex(n, H, F)$ reduces to the classical Tur\'an number $\ex(n, F)$.

The first result of this type is due to Zykov \cite{Zykov}, who in 1949 determined $\ex(n, K_{r}, K_{k})$.
For a graph $G$, let $C_{t}(G)$ be the set of all $t$-cliques in $G$, and let $c_{t}(G) = |C_{t}(G)|$.
Recently, Gerbner \cite{Gernber,Gernber2} determined $\ex(n,K_{t},kK_{r+1})$ for sufficiently large $n$.
For further related results, see \cite{Gernber3}.

\begin{thm}[\cite{Gernber,Gernber2}]\label{Gernb}
Let $r\geq t\geq2 $ and $k\geq1$. For sufficiently large $n$,
$$\ex(n,K_{t},kK_{r+1})=c_{t}(K_{k-1}\vee T_{r}(n-k+1) ).$$
\end{thm}

In recent years, increasing attention has been devoted to the spectral Tur\'an problem.
The \emph{spectral radius} of a graph $G$, denoted by $\rho(G)$, is the largest eigenvalue of its adjacency matrix $A(G)$.
Nikiforov \cite{Nikiforov2} proved that if $G$ is a $K_{r+1}$-free graph on $n$ vertices, then $\rho(G)\leq \rho(T_{r}(n))$,
with equality if and only if $G=T_{r}(n)$.
Subsequently, Ni, Wang, and Kang \cite{Ni} determined the spectral extreme for the forbidding $kK_{r+1}$, as stated below.
Further related results can be found in \cite{Chen,Fang,Li,Liu.L,Liu.R,Nikiforov,Re,Cio2,Zhai}

\begin{thm}[\cite{Ni}]\label{NiK}
For $k\geq 2 $,  $r\geq 2 $, and $n$ be sufficiently large.
If $G$ is a $kK_{r+1}$-free graph on $n$ vertices,
then $$\rho(G) \leq \rho(K_{k-1}\vee T_{r}(n-k+1)),$$
with equality if and only if $G \cong K_{k-1}\vee T_{r}(n-k+1)$.
\end{thm}

Using the fact that $\rho(G) \geq \frac{2e(G)}{n}$ and $e(T_{r}(n)) = \big\lfloor \frac{n}{2} \rho(T_{n,r}) \big\rfloor$,
one can derive the edge bound in Theorem \ref{edge-kK} from the spectral bound in Theorem \ref{NiK}.
This connection highlights the strength of spectral methods in extremal graph theory and motivates the study of spectral analogues of generalized Tur\'an problems.

To this end, we adopt a framework based on hypergraph tensor representations, introduced by Cooper and Dutle \cite{Cooper} and further developed by Liu and Bu \cite{Liu2}.
We first recall some basic facts about tensors.
An order $k$ and dimension $n$ tensor $\A$ over $\mathbb{C}$ is a multidimensional array $\A=(a_{i_1 i_2 \ldots, i_k})$ with $n^k$ complex entries, where $i_j \in [n]:=\{1,2,\ldots,n\}$ for each $j \in [k]$.
The notions of eigenvalues and eigenvectors for a tensor were introduced independently by Qi \cite{Qi1} and Lim \cite{Lim}, as follows.

\begin{definition}[\cite{Qi1,Lim}]
Let $\A = (a_{i_{1} i_{2} \cdots i_{k}})$ be an order $k \geq 2$ and dimension $n$ tensor.
If there exist a number $\lambda \in \mathbb{C}$ and a vector $x = (x_{1}, x_{2}, \ldots, x_{n}) \in \mathbb{C}^n$ such that
$$\A x^{k-1} = \lambda x^{[k-1]},$$
then $\lambda$ is called an eigenvalue of $\A$, and $x$ is an eigenvector of $\A$ corresponding to $\lambda$.
Here $\A x^{k-1} \in \mathbb{C}^n$ is defined by
$$(A x^{k-1})_j= \sum_{i_{2}, \ldots, i_{k} \in [n]} a_{j i_{2} \cdots i_{k}} x_{i_{2}} \cdots x_{i_{k}}, ~ j \in [n],$$
and $x^{[k-1]}:=(x_1^{k-1},\ldots, x_n^{k-1})$.
\end{definition}

The \emph{spectral radius} of $\A$, denoted by $\rho(\A)$, is defined as the maximum modulus of the eigenvalues of $\A$.
A tensor is called \emph{symmetric} if its entries are invariant under any permutation of the indices.
Let $\mathbb{R}_{+}^n$  be the set of nonnegative real vectors in $\mathbb{R}^n$.
Qi \cite{Qi2} established the following result for nonnegative symmetric tensors, where the $\ell_k$-norm of a vector $x=(x_{1}, x_{2}, \ldots, x_{n})$ is defined by $\|x\|_k=\left(\sum_{i=1}^n |x_i|^k \right)^{1/k}$.

\begin{lem}[\cite{Qi2}]\label{QiSpec}
Let $\A$ be an order $k \geq 2$ and dimension $n$ nonnegative symmetric tensor.
Then
$$\rho(\A) = \max \bigg\{ x^\top \A x^{k-1}:=\sum_{i_{1}, \cdots, i_{k} \in [n]} a_{i_{1} i_{2} \cdots i_{k}} x_{i_{1}} x_{i_{2}} \cdots x_{i_{k}}: x\in \mathbb{R}_{+}^n, \|x\|_k=1\bigg\}.$$
\end{lem}

Recently, Liu and Bu \cite{Liu2} introduced the $t$-clique tensor of a graph.

\begin{definition}[\cite{Liu2}]\label{t-cliq}
Let $G$ be a graph on $n$ vertices. The $t$-clique tensor $\A_{t}(G)=(a_{i_{1}i_{2}\cdots i_{t}})$ is defined by
$$
a_{i_{1}i_{2}\cdots i_{t}}=
\begin{cases}
\frac{1}{(t-1)!}, & \text{if } \{i_{1}, \ldots,i_{t}\} \in C_{t}(G), \\
0, & \text{otherwise}.
\end{cases}
$$
\end{definition}

Observe that the collection of all $t$-cliques in $G$ forms a $t$-uniform hypergraph $H_t(G)$ on the same vertex set as $G$, called the \emph{$t$-clique hypergraph} of $G$, where $\{i_1,\ldots,i_t\}$ is a hyperedge if and only if it induces a $t$-clique in $G$.
Thus, $\A_t(G)$ coincides with the adjacency tensor of $H_t(G)$ (see \cite{Cooper}).
The \emph{$t$-clique spectral radius} of $G$ is defined as $\rho_{t}(G):=\rho(\A_t(G))$.
In particular, when $t=2$, $\A_{t}(G)$ reduces to the adjacency matrix of $G$, and $\rho_{t}(G)$ is the usual spectral radius of $G$.

By a slight abuse of notation, we use the vertex set of a $t$-clique to represent the clique itself.
By Definition \ref{t-cliq} and Lemma \ref{QiSpec}, we obtain
$$\rho_{t}(G)=\max \bigg\{t\sum_{\{i_{1}, \ldots,i_{t}\}\in C_{t}(G)}x_{i_{1}}\cdots x_{i_{t}}:
x \in \mathbb{R}_{+}^n, \|x\|_t = 1 \bigg\}.$$
Liu and Bu \cite{Liu2} also determined the maximum value of $\rho_{t}(G)$ among $K_{r+1}$-free graphs with $n$ vertices, which can be viewed as a spectral analogue of $\ex(n,K_{r},kK_{r+1})$.

\begin{thm}[\cite{Liu2}]\label{LiuClq}
Let $G$ be a $K_{r+1}$-free graph on $n$ vertices. Then
$$\rho_{r}(G) \leq \rho(T_{r}(n)),$$
with equality if and only if $G\cong T_{r}(n).$
\end{thm}

Yu and Peng \cite{Peng} extended this result, which can be regarded as a spectral counterpart of Zykov's theorem \cite{Zykov} on generalized Tur\'an numbers.

\begin{thm}\label{PengClq}
Let $G$ be a $K_{r+1}$-free graph on $n$ vertices. Then for any $2\leq t \leq r$,
$$\rho_{t}(G) \leq \rho_{t}(T_{r}(n)),$$
with equality if and only if $G\cong T_{r}(n).$
\end{thm}

For further results on spectral versions of the generalized Tur\'an problem, we refer the reader to \cite{Peng2, Peng3, Liu3}.
In this paper, we determine the maximum $t$-clique spectral radius among all $kK_{r+1}$-free graphs on $n$ vertices.
Our results can be viewed as a spectral analogue of Theorem \ref{Gernb}.

\begin{thm}\label{Main}
Let $G$ be a $kK_{r+1}$-free graph on $n$ vertices.
Then, for integers $r\geq t\geq2$, $k\geq 2$, and sufficiently large $n$,
$$\rho_{t}(G) \leq \rho_{t}(K_{k-1}\vee T_{r}(n-k+1)),$$
with equality if and only if $G\cong K_{k-1}\vee T_{r}(n-k+1)$.
\end{thm}

\begin{remark} \rm
Theorem \ref{Main} serves as a spectral counterpart of Theorem \ref{Gernb}.
In particular, when $t=2$, it reduces to Theorem \ref{NiK}.
So, Theorem \ref{Main} provides a unified spectral extension of these classical extremal results.
\end{remark}

The remainder of this paper is organized as follows.
In Section 2, we introduce notation and collect several auxiliary lemmas.
In Section 3, we first prove the $t$-clique connectivity of the extremal graph,
and then derive some structural properties that characterize it.
We subsequently analyze the entries of a corresponding eigenvector.
These ingredients are combined to complete the proof of Theorem \ref{Main}.

\section{Preliminaries}\label{sec2}
Let $G$ be a graph with vertex set $V(G)$ and edge set $E(G)$.
For a vertex $v\in V(G)$ and $S\subseteq V(G)$,
let $N_{G}(v)$ denote the neighborhood of $v$ in $G$, and let $d_{G}(v) = |N_{G}(v)|$ be its degree.
For convenience, we write $N_{S}(v):=N_G(v) \cap S$ and $d_{S}(v) = |N_{S}(v)|$.
For $S\subseteq V(G)$, let $G[S]$ denote the subgraph of $G$ induced by $S$,
and let $G-S$ denote the graph obtained by deleting the vertices in $S$ together with all incident edges.

\begin{definition}[\cite{Yang1}]
 Let $\A$ be an order $m$ and dimension $n$ tensor.
 If there exists a nonempty proper subset $I \subseteq [n]$ such that $a_{i_1 i_2 \cdots i_m} = 0$ whenever $i_1 \in I$ and $\{i_2, \ldots, i_m\} \not\subseteq I$, then $\A$ is called weakly reducible. Otherwise, $\A$ is called weakly irreducible.
 \end{definition}

\begin{lem}[\cite{Yang2,Friedland})]
Let $\A$ be an order $m$ and dimension $n$ nonnegative tensor.
Then its spectral radius $\rho(\A)$ is an eigenvalue of $\A$ corresponding to a nonnegative eigenvector.
If $\A$ is weakly irreducible, then $\rho(\A)$ is the unique eigenvalue of $\A$ corresponding to a unique positive eigenvector up to scaling.
\end{lem}

For two real tensors $\A$ and $\mathcal{B}$ of same order and dimension,
we write $\mathcal{B} < \A$ if $\A - \mathcal{B}$ is a nonnegative and nonzero tensor.

\begin{lem}[\cite{Khan})]\label{Khan-rho}
Let $\A$ and $\mathcal{B}$ be two nonnegative tensors of same order and dimension.
If $\A < \mathcal{B}$ and $\mathcal{B}$ is weakly irreducible,
then $\rho(\A) < \rho(\mathcal{B})$.
\end{lem}

\begin{lem}[\cite{Shao})]\label{Shao-diag}
 Let $\A$ be a weakly reducible tensor.
 Then, $\A$ is permutationally similar to a lower triangular block tensor such that each diagonal block tensor is weakly irreducible.
 Furthermore, one of these diagonal block tensors has spectral radius $\rho(\A)$.
\end{lem}

If, in addition, the tensor $\A$ in Lemma \ref{Shao-diag} is symmetric, then $\A$ is permutationally similar to a diagonal block tensor.

Liu and Bu \cite{Liu2} introduced the notions of $t$-clique walk and $t$-clique connectedness.
A \emph{$t$-clique walk} is a sequence of $t$-cliques such that every two consecutive $t$-cliques share at least one vertex.
The graph $G$ is said to be \emph{$t$-clique connected} if every pair of vertices $u, v \in V(G)$ can be connected by a $t$-clique walk.
The following lemma connects this notion with tensor irreducibility.

\begin{lem}[\cite{Liu2})]\label{Liu-conn}
For a graph $G$, its $t$-clique tensor $\A_t(G)$ is weakly irreducible if and only if $G$ is
$t$-clique connected.
\end{lem}

The following lemma connects $t$-clique spectral radius with the number of $t$-cliques.

\begin{lem}[\cite{Liu2})]\label{Spec-Clq}
Let $G$ be a graph on $n$ vertices. Then
$$\rho_{t}(G) \geq \frac{t}{n}c_{t}(G).$$
Furthermore, equality holds if every vertex of $G$ is contained in the same number of $t$-cliques.
\end{lem}

Liu and Bu \cite{Liu} established the following spectral analogue of the Erd\H{o}s-Simonovits stability theorem for $t$-clique tensors, where $\chi(G)$ denotes the chromatic number of a graph $G$.

\begin{lem}[\cite{Liu})]\label{Stability}
Let $H$ be a graph with $\chi(H) = r+1 > t \geq 2$. For every $\varepsilon > 0$, there exist $\delta > 0$ and $n_{0}$ such that if $G$ is an $H$-free graph on $n \geq n_{0}$ vertices and
$$\rho_t(G) \geq \left( \binom{r-1}{t-1} \left( \frac{1}{r} \right)^{t-1} - \delta \right) n^{t-1},$$
then $G$ can be obtained from $T_r(n)$ by adding and deleting at most $\varepsilon n^2$ edges.
\end{lem}

If a graph $G$ is structurally close to a Tur\'an graph, then it inherits several useful structural properties.
Suppose that $V(G)$ has a partition $V(G)=V_{1}\cup\cdots\cup V_{r}$.
Let $e(V_i)$ denote the number of edges in the induced subgraph $G[V_i]$, and for $i \ne j$, let $e(V_i,V_j)$ denote the number of edges of $G$ with one endpoint in $V_i$ and the other in $V_j$.
The following result is due to Ni, Wang, and Kang \cite{Ni}.

\begin{lem}[\cite{Ni}]\label{Ni-Turan}
Let $G$ be a graph obtained from $T_r(n)$ by adding and deleting at most $\varepsilon n^2$ edges.
Then there exists a partition $V(G)=V_{1}\cup\cdots\cup V_{r}$ such that the number of crossing edges of $G$ (i.e. $\sum_{1\le i<j\le r}e(V_{i},V_{j})$) is maximized and $\sum_{i=1}^{r}e(V_{i})\leq\varepsilon n^{2}$.
Moreover, for such a partition, the following hold:

\begin{enumerate}
\item[\rm (a)] For each $i\in[r]$, $\tfrac{n}{r}-3\sqrt{\varepsilon}n<|V_{i}|<\tfrac{n}{r}+3\sqrt{\varepsilon}n.$

\item[\rm (b)] Let $W:=\cup_{i=1}^{r}\{v\in V_{i}:d_{V_{i}}(v)\geqslant 2\theta n\},$
where  $\theta$ is a sufficiently small constants with $\sqrt{\varepsilon} \le  \theta<\frac{1}{20kr^{4}(r+1)}$.
Then, $|W|\leq \theta n$.

\item[\rm (c)] Let $L:=\{v\in V(G):d(v)\leqslant(1-\tfrac{1}{r}-\varepsilon_{1})n\}$, where
$\varepsilon_{1}$ is a sufficiently small constant with $\sqrt{\varepsilon}<\varepsilon_{1} \ll \theta$.
Then, $|L|\le \varepsilon_{2}n$, where $\varepsilon_{2}\ll\varepsilon_{1}$ is a sufficiently small constant
satisfying $\varepsilon-\varepsilon_{1}\varepsilon_{2}+\tfrac{r-1}{2r}\varepsilon_{2}^{2}<0$.

\item[\rm (d)] For each $i\in[r]$, there exists an independent set $I_{i}\subseteq V_{i}\setminus (W\cup L)$ such that
$|I_{i}|\geq |V_{i}\setminus (W\cup L)|-2(k-1).$

\item[\rm (e)] $|W\setminus L| \le k-1$.
\end{enumerate}
\end{lem}

Next, consider a simple graph $G$ with matching number $\mu(G)$ and maximum degree $\Delta (G)$.
For fixed integers $m$ and $\Delta$, define
$$f(m,\Delta) = \max\{ e(G) : \mu(G) \leq m,\ \Delta(G) \leq \Delta \}.$$
The following result was established by Chv\'atal and Hanson \cite{Chv}.

\begin{lem}[\cite{Chv}]\label{chv-md}
For all integers $m \geq 1$ and $\Delta \geq 1$,
$$f(m,\Delta) = \Delta m + \left\lfloor \frac{\Delta}{2} \right\rfloor
\left\lfloor \frac{m}{\lceil \Delta/2 \rceil} \right\rfloor
\leq (\Delta +1) m.$$
\end{lem}

Finally, we recall a simple inequality for finite sets.

\begin{lem}[\cite{Cio}]\label{set-cap}
Let $V_{1},\ldots ,V_{n}$ be finite sets. Then
$$|V_{1}\cap \dots \cap V_{n}|\geq \sum_{i = 1}^{n}|V_{i}| - (n - 1)\left|\textstyle\bigcup_{i = 1}^{n}V_{i}\right|.$$
\end{lem}

\section{Proof of Theorem \ref{edge-kK}}
In this section, we assume throughout that $r\geq t\geq2,$ $k\geq 2$, and $n$ is sufficiently large, and that \emph{$G$ is an extremal graph with the maximum $t$-clique spectral radius among all $kK_{r+1}$-free graphs on $n$ vertices}.
Our goal is to show that $G \cong K_{k-1}\vee T_{r}(n-k+1)$.

The proof proceeds as follows.
We first show that $G$ is $t$-clique connected (Lemma \ref{t-clq-conn}), which ensures the existence of a positive eigenvector $x$ of $\A_t(G)$ corresponding to $\rho_t(G)$.
Next, we establish a lower bound for $\rho_t(G)$ (Lemma \ref{lowerbound}).
This allows us to apply the stability theorem (Lemma \ref{Stability}) to deduce that $G$ differs from $T_r(n)$ by at most $\varepsilon n^2$ edges.
Consequently, $V(G)$ admits a partition as in Lemma \ref{Ni-Turan}, together with a set $W$ of high-degree vertices and a set $L$ of low-degree vertices.
By investigating the local structural properties of $G$ (Lemmas \ref{v-neigh} and \ref{v-indeg}) and the lower bound for the second largest entry of the vector $x$ under normalization (Lemma \ref{2nd-ev}), we then show that $L=\emptyset$  (Lemma \ref{LO}) and $W$ is a dominant set of size $k-1$ (Lemmas \ref{W} and \ref{W-dom}).
It follows that $G$ is a join of a complete graph $K_{k-1}$ and an $r$-partite graph. Finally, by exploiting the maximality of $\rho_t(G)$, we complete the proof.

We begin by establishing the $t$-clique connectedness of $G$.

\begin{lem}\label{t-clq-conn}
$G$ is $t$-clique connected. Consequently, the $t$-clique tensor $\A_t(G)$ is weakly irreducible.
\end{lem}

\begin{proof}
Suppose, to the contrary, that $G$ is not $t$-clique connected; equivalently, the $t$-clique hypergraph $H_G$ is not connected.
First note that that $G$ must contain at least one $t$-clique.
Otherwise, by adding edges to $G$, we obtain a $kK_{r+1}$-free graph $G'$ that contains a $t$-cliques, yielding $\rho_t(G')> \rho_t(G)=0$, a contradiction.
By Lemma \ref{Shao-diag}, there exists a subgraph $G_{1}$ of $G$ such that the associated $t$-clique hypergraph $H_{G_1}$ is a connected component of $H_G$ and  satisfies $\rho_t(G_{1}) = \rho_t(G)$.
Let $u \in V(G)\setminus V(G_1)$, and let $C_{1}$ be a $t$-clique of $G_1$ maximizing the number of neighbors of $u$ in $C_1$.
Let $v_{1},\dots,v_{s}$ be the non-neighbors of $u$ in $C_{1}$.
Since $H_{G_1}$ is a connected component of $H_G$, we have $2 \leq s \leq t$.
Indeed, if $s \le 1$, then $u$ together with its neighbors in $C_1$ would form a $t$-clique intersecting $G_1$, contradicting the choice on $G_1$.

Construct a new graph $G'$ from $G$ by joining $u$ and each of $v_{1}, \dots, v_{s-1}$.
We claim that $G'$ remains $kK_{r+1}$-free.
Otherwise, $G'$ contains a new copy of $K_{r+1}$ involving $u$ and at least one vertex from $\{v_{1}, \dots, v_{s-1}\}$.
Let this new clique be $\{u\} \cup X$,
where $X$ is an $r$-clique in $G$ with $X \cap \{v_{1}, \dots, v_{s-1}\} \neq \emptyset$.

\textbf{Case 1.} $X\cap (V(G)\setminus V(G_{1})) \neq \emptyset$.
Assume without loss of generality that $v_1 \in X \cap \{v_{1}, \dots, v_{s-1}\}$.
Take any vertex $w\in X\cap (V(G)\setminus V(G_{1}))$.
Then $v_1$ and $w$ lie in a common $r$-clique $X$ of $G$, which contains a $t$-clique intersecting both $V(G_1)$ and its complement; contradicting the assumption on $G_1$.

\textbf{Case 2}. $X\subseteq V(G_{1})$.
Then $X$ induces a $K_{r}$ in $G_{1}$, and $u$ has at most $s-1$ non-neighbors in $X$.
Since $r \geq t$, we can choose a $t$-subset $Y$ from $X$.
Surely, $Y$ induces a  $t$-clique $C_2$ in $G_{1}$.
However,
$$|N_{C_{2}}(u)| \geq t-s+1 = |N_{C_{1}}(u)|+1,$$
contradicting the maximuality of $C_1$.

Thus, $G'$ is $kK_{r+1}$-free.
Moreover, the vertices
$u, v_{1}, \dots, v_{s-1}$ together with the neighbors of $u$ in $C_1$ form
a new $t$-clique in $G'$.
Hence, the $t$-clique hypergraph $H_{G'}$ has a connected component that contains $H_{G_1}$ as a proper subhypergraph.
By Lemma \ref{Khan-rho} and Lemma \ref{Liu-conn},
$$\rho_t(G') > \rho_t(G_{1})=\rho_t(G),$$
a contradiction.
\end{proof}

Next, we derive a lower bound for $\rho_{t}(G)$.

\begin{lem}\label{lowerbound}
$$\rho_{t}(G)\geq  \binom{r-1}{t-1}\left(\frac{n}{r}\right)^{t-1}+\Theta(n^{t-2}).$$
\end{lem}

\begin{proof}
Let $G' = K_{k-1} \vee T_r(n-k+1)$. Clearly, $G'$ is $kK_{r+1}$-free.
Since $G$ has the maximum $t$-clique spectral radius,
we have $\rho_t(G) \geq \rho_t(G')$.
We estimate the number of $t$-cliques in $G'$.
A direct calculation gives
$$c_t(G') = \binom{r}{t} \left(\frac{n}{r}\right)^t + (k-1)\left(\binom{r}{t-1}-\binom{r-1}{t-1}\right)\left(\frac{n}{t}\right)^{t-1}+O(n^{t-2}),$$
where the leading contribution comes from the $t$-cliques entirely contained in $T_r(n-k+1)$.
By Lemma \ref{Spec-Clq}, we obtain
$$\rho_t(G) \geq \frac{t}{n} c_t(G')
\geq  \binom{r-1}{t-1}\left(\frac{n}{r}\right)^{t-1}+\Theta(n^{t-2}),$$
as claimed.
\end{proof}

By Lemma \ref{lowerbound}, for any $\delta >0$, there exists $n_0$ such that for all $n \ge n_0$,
$$\rho_t(G) > \left( \binom{r-1}{t-1} \left( \frac{1}{r} \right)^{t-1} - \delta \right) n^{t-1}.$$
Since $\chi(kK_{r+1}) = r+1 > 2$, Lemma \ref{Stability} implies that $G$ can be obtained from $T_r(n)$ by adding and deleting at most $\varepsilon n^2$ edges.
Hence, by Lemma \ref{Ni-Turan}, there exists a partition $V(G)=V_{1}\cup\cdots\cup V_{r}$ such that
$\sum_{1\leqslant i<j\leqslant r}e(V_{i},V_{j})$ is maximized
and $\sum_{i=1}^{r}e(V_{i})\leqslant\varepsilon n^{2}$.
Define
$$
W := \textstyle\bigcup_{i=1}^{r} \{ v \in V_{i} : d_{V_{i}}(v) \geq 2\theta n \},
~ L := \bigl\{ v \in V(G) : d_{G}(v) \leq \bigl(1 - \tfrac{1}{r} - \varepsilon_{1}\bigr)n, \bigr\}
$$
where $\theta$ is sufficiently small with $\theta < \frac{1}{20kr^{4}(r+1)}$, and
$\sqrt{\varepsilon} < \varepsilon_{1} \ll \theta$.

We now investigate the local structural properties of $G$.
Let $v \in V_i$.
We will show that the neighborhood of $v$ in $\bigcup_{j \ne i} V_j \setminus (W \cup L)$ is $kK_r$-free (Lemma \ref{v-neigh}), and that if $v \in V_i \setminus (W \cup L)$, then $v$ has at most $(k-1)(r+1)$ neighbors in $V_i \setminus (W \cup L)$ (Lemma \ref{v-indeg}).

\begin{lem}\label{v-neigh}
Let $\hat{V}_i=\bigcup_{j \ne i} V_j$ for each $i \in [r]$.
Then, for any $v\in V_{i}$,
the induced subgraph $G[N_{\hat{V}_i\setminus (W\cup L)}(v)]$ is $kK_{r}$-free.
\end{lem}

\begin{proof}
Suppose, to the contrary, that there exists a vertex $v \in V_{i}$
such that $G[N_{\hat{V}_i\setminus (W\cup L)}(v)]$ contains $k K_r$.
Fix one copy of $K_{r}$ with vertex set $u_{1},\cdots,u_{r}$.
Since the vertex set $V(G[N_{\hat{V}_i\setminus (W\cup L)}(v)])$ is partitioned into $r-1$ parts, the set $\{u_{1},\cdots,u_{r}\}$ contains two vertices from the same part.
Without loss of generality, assume that $v \in V_{1}$ and
$u_{1},u_{2} \in V_{2}$.
We show that $u_{1},\ldots,u_{r}$ have a large common neighborhood.

For each $u_{i}$ ($1\leq i\leq r$), let $s_i$ be the index such that $u_{i} \in V_{s_i}$.
Since $u_i \notin L$, we have
\begin{equation}\label{NGui}
|N_{G}(u_{i})|=|N_{G}(u_{i})\cap V_{s_i}|+\sum_{j\neq s_i}|N_{G}(u_i)\cap V_{j}|
> \bigl( 1-\tfrac{1}{r}-\varepsilon_{1}\bigr)n.
\end{equation}
As $u_{i}\notin W$, it follows that $|N_{G}(u_{i})\cap V_{s_i}|\leq 2\theta n$.
Hence,
\begin{equation}\label{NGuiVj}
\sum_{j\neq s_i}|N_{G}(u_{i})\cap V_{j}|
\geq\frac{r-1}{r}n-\varepsilon_{1}n-2\theta n.
\end{equation}
On the other hand, by Lemma \ref{Ni-Turan} $(a)$,
$\tfrac{n}{r}-3\sqrt{\varepsilon}n<|V_{i}|<\tfrac{n}{r}+3\sqrt{\varepsilon}n$ for any $i\in[r]$, and thus
\begin{equation}\label{SumVj}
\sum_{j\neq s_i}|V_{j}|=n-|V_{s_i}|
\leq \frac{r-1}{r}n-3\sqrt{\varepsilon}n.
\end{equation}
Subtracting (\ref{NGuiVj}) from (\ref{SumVj}) gives
\begin{equation}\label{u1u2NonN}
\sum_{j\neq s_i}(|V_{j}|-|N_{G}(u_{i})\cap V_{j}|)
\leq (\varepsilon_{1}-3\sqrt{\varepsilon}+2\theta)n
=O(\theta n),
\end{equation}
where the equality follows from
$\sqrt{\varepsilon}<\varepsilon_{1} \ll \theta$.
Thus, for each fixed index $j\neq s_i$, the number of non-neighbors of $u_{i}$ in $V_{j}$ is $O(\theta n)$.
In particular, for $j\neq 2$, by \eqref{u1u2NonN}, we have
$$|(V_j\setminus N_{G}(u_{1}))\cup(V_{j}\setminus N_{G}(u_{2}))|
\leq|V_{j}\setminus N_{G}(u_{1})|+|V_{j}\setminus N_{G}(u_{2})|
=O(\theta n).$$

For $i\in[r]$, let $A_{i}:=V(G)\setminus N_{G}(u_{i})$ be
the set of non-neighbors of $u_{i}$ in $G$.
By \eqref{NGui},
$$ |A_{i}|\leq n-|N_{G}(u_{i})|\leq \frac{n}{r}+\varepsilon_{1}n.$$
Let $B_{i}:=V_{s_i}\setminus N_{G}(u_{i})$ be the set of non-neighbors of $u_{i}$ within $V_{s_i}$.
Since $B_{1}\cup B_{2}\subseteq V_{2}$, it follows that
\begin{equation}\label{B1B2}
|B_{1}\cup B_{2}|\leq |V_{2}|<\frac{n}{r}+3\sqrt{\varepsilon}n.
\end{equation}
Thus, by \eqref{B1B2} and \eqref{u1u2NonN}, we have
$$
\begin{aligned}
|A_{1}\cup A_{2}| &=|B_{1}\cup B_{2}|+\sum_{j\neq 2}|(V_{j}\setminus N_{G}(u_{1}))\cup(V_{j}\setminus N_{G}(u_{2}))|\\
&\leq\left(\frac{n}{r}+3\sqrt{\varepsilon}n\right)+(r-1)\cdot O(\theta n)\\
&=\frac{n}{r}+O(\theta n).
\end{aligned}
$$
Therefore,
$$
\begin{aligned}
\left|\textstyle\bigcup_{i=1}^{r}A_{i}\right| &\leq |A_{1}\cup A_{2}|+\sum\limits_{j=3}^{r}|A_{j}|\\
&\leq \frac{n}{r}+O(\theta n)+(r-2)\left(\frac{n}{r}+\varepsilon_{1}n\right)\\
&= \frac{r-1}{r}\,n+O(\theta n).
\end{aligned}
$$
Hence,
$$\left| \textstyle\bigcap_{i=1}^{r} N_{G}(u_{i})\right|
= n-\left| \textstyle\bigcup_{i=1}^{r} A_{i}\right|
\geq \frac{n}{r} - O(\theta n).$$
For sufficiently large $n$ and sufficiently small $\theta$, each such $K_r$ has a large common neighborhood.

Thus, if $G[N_{\hat{V}_i\setminus (W\cup L)}(v)]$ contains $k$ vertex-disjoint copies of $K_{r}$, we can greedily select one common neighbor for each to obtain $k$ vertex-disjoint copies of $K_{r+1}$ in $G$,
contradicting that  $G$ is $kK_{r+1}$-free.
\end{proof}

\begin{lem}\label{v-indeg}
For any $i \in [r]$ and any $v\in V_{i}\setminus (W\cup L)$,
$d_{V_{i}\setminus (W\cup L)}(v)\leq (k-1)(r+1)$.
\end{lem}

\begin{proof}
Suppose, to the contrary, that there exists a vertex $v \in V_{i}\setminus (W\cup L)$ for some $i \in [r]$
such that $d_{V_{i}\setminus (W\cup L)}(v)\geq (k-1)(r+1)+1$.
Without loss of generality, assume $v \in V_{1} \setminus (W\cup L)$.

Let $G'$ be the graph obtained from $G$ by adding all edges between $v$ and the vertices of $V_{1}\setminus (W\cup L)$, namely,
$E(G')=E(G)\cup\{vw: w\in V_{1}\setminus (W\cup L)\}.$
We first claim that $G'$ is still $kK_{r+1}$-free.
Otherwise, suppose that $G'$ contains $k$ vertex-disjoint copies of $K_{r+1}$, say $C_{1}, C_{2}, \dots, C_{k}$.
Since any new $K_{r+1}$ in $G'$ must use at least one added edge, it must contain the vertex $v$.
Thus, exactly one of these cliques, say $C_{1}$, contains $v$.

Let $G''=G-\bigcup_{j=2}^{k}V(C_{j})$,
and write $V(G'')=V_{1}'\cup\cdots\cup V_{r}'$ according to the original partition,  where $v \in V_{1}'$.
Thus, $$d_{V_{1}'\setminus (W\cup L)}(v)\geq d_{V_{1}\setminus (W\cup L)}(v)-(k-1)(r+1)\geq 1.$$
Let $u$ be a neighbor of $v$ in $G''[V_{1}' \setminus (W\cup L)]$.
We show that $G''$ contains a $K_{r+1}$ containing $u$ and $v$, together with one vertex from each of $V'_2, \ldots, V'_r$.

We construct such $K_{r+1}$ iteratively.
For any vertex $w \in V_{1}' \setminus (W\cup L)$, we have
$$d_{G''} (w)\geq \left(1-\frac{1}{r}-\varepsilon_{1}\right)n-(k-1)(r+1).$$
Recall that $|V_{i}'|\leq |V_{i}|<n/r+3\sqrt{\varepsilon}n$ for any $i\in[r]$.
Hence,
$$
\begin{aligned}
d_{V_{2}'}(w) &\geq d_{G''}(w)-d_{V_{1}'}(w)-\sum_{s =3}^{r} |V_{s}'|\\
&\geq \left(1-\frac{1}{r}-\varepsilon_{1}\right)n-(k-1)(r+1)-2\theta n- (r-2)\left(\frac{n}{r}+3\sqrt{\varepsilon}n\right)\\
&\geq \frac{n}{r}-g(n),
\end{aligned}
$$
where $g(n) = (3r-6)\sqrt{\varepsilon}n + 2\theta n + \varepsilon_{1}n + (k-1)(r+1).$
Thus, we have
$$
\begin{aligned}
|N_{V_{2}'}(u) \cap N_{V_{2}'}(v) \setminus (W\cup L)|
&\geq d_{V_{2}'}(u) +d_{V_{2}'}(v) -|N_{V_{2}^{'} }(u)\cup N_{V_{2}' }(v)|-|W|-|L|\\
&\geq 2\left(\frac{n}{r} - g(n)\right) - \left(\frac{n}{r} + 3\sqrt{\varepsilon}n\right) - \theta n - \varepsilon_{2}n\\
&= \frac{n}{r} - \left(2g(n) + 3\sqrt{\varepsilon}n + \theta n+\varepsilon_{2}n\right)\\
&> 1,
\end{aligned}
$$
for sufficiently small $\varepsilon,\varepsilon_{2},\theta$ and large $n$.
Thus, $u$ and $v$ have a common neighbor in $V_{2}' \setminus (W\cup L)$, which together form a $K_{3}$.

Proceeding inductively, for every integer $s$ with $3 \leq s \leq r$, suppose we have constructed a $K_{s}$ with vertex set
$S=\{u,v,w_{1},\ldots,w_{s-2}\}$,
where $w_{i} \in V_{i+1}' \setminus (W\cup L)$ for $i \in [s-2]$.
Then, by Lemma \ref{set-cap},
\begin{equation}\label{S-neigh}
\begin{split}
\bigg| \textstyle\bigcap_{w \in S} N_{V_{s}'}(w) \setminus (W\cup L) \bigg|
&\geq \sum_{w \in S} d_{V_{s}'}(w)-(|S|-1) \bigg| \textstyle\bigcup_{w \in S} N_{V_{s}'}(w) \bigg| - |W|- |L|\\
&\geq s\left(\frac{n}{r} - g(n)\right)-(s-1)\left(\frac{n}{r} + 3\sqrt{\varepsilon}n\right) - \theta n -\varepsilon_{2}\\
&= \frac{n}{r} - [ s g(n) + 3(s-1)\sqrt{\varepsilon}n + \theta n +\varepsilon_{2}n]\\
&> 1.
\end{split}
\end{equation}
Thus we can choose $w_{s-1} \in V_{s}' \setminus (W\cup L)$ adjacent to all vertices of $S$.
Then $S \cup \{w_{s-1}\}$ induces a $K_{s+1}$, extending the above clique $K_s$.
After $r-1$ steps, we can find a $K_{r+1}$ in $G''$.
By definition, $G$ contains a $kK_{r+1}$, a contradiction.
Therefore, $G'$ is $kK_{r+1}$-free.

Finally, by Lemma \ref{Ni-Turan}, we have
$$
\begin{aligned}
|V_{1} \setminus (W\cup L)| -d_{V_{1} \setminus (W\cup L)}(v)
& \geq |V_{1}|-|W|-|L| - d_{V_1}(v)\\
& >\left(1/r-3\sqrt{\varepsilon}-3\theta-\varepsilon_{2}\right)n.
\end{aligned}
$$
Hence, $G'$ is obtained from $G$ by adding a large number of edges incident to the vertex $v$.
As in the above argument, $G'$ produces new $t$-cliques that contain some of the added edges.
Since  $G$ is $t$-clique connected (Lemma \ref{t-clq-conn}), so is $G'$, and hence
by Lemma \ref{Khan-rho}, $\rho_{t}(G')> \rho_{t}(G)$; a contradiction.
Thus, $d_{V_{i}\setminus (W\cup L)}(v)\leq (k-1)(r+1)$ for all $v \in V_i \setminus (W \cup L)$.
\end{proof}

Actually, since both $\varepsilon_{2},\varepsilon$ and $\theta$ are sufficiently small, the lower bound in \eqref{S-neigh} can be strengthened to
$$\bigg| \textstyle\bigcap_{w \in S} N_{V_{s}'}(w) \setminus (W\cup L) \bigg|
\geq n/r- [ s g(n) + 3(s-1)\sqrt{\varepsilon}n + \theta n +\varepsilon_{2}n]   > n/(r+1).$$
This implies that for any vertex $u \in V(G)\setminus (W\cup L)$, or any edge $uv$ contained in some $V_{i} \setminus (W\cup L)$,
its neighborhood contains a complete $(r-1)$-partite subgraph with each part of size at least $n/(r+1)$.
Arguing as in Lemma \ref{v-indeg}, we obtain the following corollary.

\begin{cor}\label{v-ed-clq}
For each $t$ with $2 \le t \le r$, every vertex $u \in V(G)\setminus (W\cup L)$ and every edge $uv \in E(G[V_{i}\setminus (W\cup L)])$ are contained in at least
$\binom{r-1}{t-1}\left(\frac{n}{r+1}\right)^{t-1}$ copies of $K_{t+1}$.
\end{cor}

By Lemma \ref{t-clq-conn}, $\A_t(G)$ admits a positive eigenvector $x$ corresponding to $\rho_{t}(G)$.
By normalization, assume $x_{u_0}=\max\{x_{v} : v \in V(G)\}=1$,
and define
$$x_{v_{0}}:=\max \{x_{v}:v\in V(G)\setminus W\}.$$
Recall that $L = \{ v \in V(G) : d_{G}(v) \leq (1 - \frac{1}{r} - \varepsilon_{1})n \}$.
We next provide a lower bound for $x_{v_{0}}$ and show that $v_0 \notin L$.

For a vertex $v$ of $G$, let $C_t(v)$ denote the family of $t$-cliques of $G$ that contains the vertex $v$, and for a subset $e \subseteq V(G)$, write $x^e:=\prod_{v \in e}x_v$.

\begin{lem}\label{2nd-ev}
$$x_{v_{0}}^{t-1}> \frac{1}{(t-1)!}\left(1-\frac{t-1/2}{r}\right)^{t-1}.$$
Moreover, $v_{0}\notin L$.
\end{lem}

\begin{proof}
Since $x_{u_0}=\max\{x_{v}:v\in V(G)\}=1$,
the eigenvector equation of $\A_t(G)$ at $u_0$ yields
$$
 \begin{aligned}
\rho_{t}(G)x_{u_0}^{t-1}&=\sum_{e \in C_t(u_0)}x^{e \setminus \{u_0\}}\\
&=\sum_{e \in C_t(u_0), ~ |(e \setminus \{u_0\}) \cap W|=0}x^{e \setminus \{u_0\}}
+\sum_{j=1}^{t-1} \sum_{e \in C_t(u_0),~ |(e \setminus \{u_0\}) \cap W|=j}x^{e \setminus \{u_0\}}\\
&\leq  (n-|W|)^{t-1}x_{v_{0}}^{t-1}+\sum\limits_{j=1}^{t-1} |W|^{j}(n-|W|)^{t-1-j}.
\end{aligned}
$$
Hence,
 \begin{equation}\label{B-xv0}
 \begin{split}
 x_{v_{0}}^{t-1}&\geq  \frac{\rho_{t}(G)x_{u_0}^{t-1}-\sum\limits_{j=1}^{t-1} |W|^{j}(n-|W|)^{t-1-j}}{(n-|W|)^{t-1}} \\
 &\geq \frac{\rho_{t}(G)}{n^{t-1}}-\sum\limits_{j=1}^{t-1} |W|^{j}(n-|W|)^{-j}.
 \end{split}
 \end{equation}
By Lemma \ref{lowerbound},
 \begin{equation}\label{B-rhoG}
 \rho_{t}(G) \geq \binom{r-1}{t-1}\left(\frac{n}{r}\right)^{t-1}+
\Theta(n^{t-2}) > \frac{1}{(t-1)!}\left(1-\dfrac{t-1}{r}\right)^{t-1}n^{t-1}.
 \end{equation}
Moreover, by Lemma \ref{Ni-Turan} (c) and (e),
 \begin{equation}\label{UB-W}
 |W| = |W\cap L|+|W\setminus L| \leq |L| + k-1\leq \varepsilon_{2}n+k-1,
 \end{equation}
which implies that for sufficiently large $n$ and sufficiently small $\varepsilon_{2}$,
 \begin{equation}\label{Ser-W}
 \sum\limits_{j=1}^{t-1} |W|^{j}(n-|W|)^{-j}\leq
 (t-1)\frac{\varepsilon_{2}}{1-\varepsilon_{2}} + o(1).
 \end{equation}
Combining  (\ref{B-xv0})-(\ref{Ser-W}), we obtain
$$x_{v_{0}}^{t-1}> \frac{1}{(t-1)!}\left(1-\dfrac{t-1}{r}\right)^{t-1}-(t-1)\frac{\varepsilon_{2}}{1-\varepsilon_{2}} - o(1)> \frac{1}{(t-1)!}\left(1-\frac{t-1/2}{r}\right)^{t-1},$$
which yields the desired bound.

We now show that $v_0 \notin L$.
Applying the eigenvalue equation at $v_0$, we have
 \begin{equation}\label{ev-xv0}
 \begin{split}
\rho_{t}(G)x_{v_{0}}^{t-1}&=\sum_{e\in C_{t}(v_0)}x^{e \setminus \{v_0\}}=\sum_{\substack{e\in C_{t}(v_0),\\|(e \setminus \{v_0\})\cap (W\cup L)|=0}}x^{e \setminus \{v_0\}}
+\sum_{j=1}^{t-1}\sum_{\substack{e\in C_{t}(v_0),\\|(e \setminus \{v_0\})\cap (W\cup L)|=j}}x^{e \setminus \{v_0\}} \\
 &\leq \sum_{e \in C_{t-1}(G[N_{G}(v_{0})\setminus (W\cup L)])}x_{v_{0}}^{t-1}
+\sum\limits_{j=1}^{t-1} |(W\cup L)|^{j}(n-|(W\cup L)|)^{t-1-j}.
\end{split}
\end{equation}

Assume that $v_{0}\in V_{1}$ and set $U:=\bigcup_{i=2}^{r}V_{i}\setminus (W\cup L)$.
A $(t-1)$-clique $K_{t-1}$ is called
\textit{pure} if $V(K_{t-1}) \subset N_{U}(v_{0})$,
and \textit{mixed} if $V(K_{t-1}) \cap N_{V_{1}\setminus (W\cup L)}(v_{0}) \neq \varnothing$.
Let $c^{\mathrm{pure}}_{t-1}(v_{0})$ and $c^{\mathrm{mix}}_{t-1}(v_{0})$
denote the numbers of pure and mixed $(t-1)$-cliques, respectively.
Clearly,
\begin{equation}\label{clq-dec}
c_{t-1}(G[N_{G}(v_{0})\setminus (W\cup L)])= c^{\mathrm{pure}}_{t-1}(v_{0})+c^{\mathrm{mix}}_{t-1}(v_{0}).
\end{equation}
By Lemma \ref{v-neigh}, we know that $G[N_{U}(v_{0})]$ is $kK_{r}$-free.
Then, by Theorem \ref{Gernb},
 \begin{equation}\label{clq-pure}
c^{\mathrm{pure}}_{t-1}(v_{0})\leq c_{t-1}(K_{k-1}\vee T_{r-1}(d_{G}(v_{0})-k+1) )
=\dbinom{r-1}{t-1}\left(\frac{d_{G}(v_{0})}{r-1}\right)^{t-1}+O(n^{t-2}).
\end{equation}
By Lemma \ref{v-indeg}, we have $d_{V_{1}\setminus (W\cup L)}(v_{0})\leq (k-1)(r+1)$, and hence
\begin{equation}\label{clq-mix}
c^{\mathrm{mix}}_{t-1}(v_{0})\leq
\dbinom{d_{V_{1}\setminus (W\cup L)}(v_0)}a\sum\limits_{a=1}^{t-1}c_{t-1-a}(G[N_{U}(v_{0})])
=O(n^{t-2}).
\end{equation}
Meanwhile, Combining  (\ref{UB-W})  and  Lemma \ref{Ni-Turan} (c), we have $ |W\cup L|\leq  |W |+|L|\leq2\varepsilon_{2}n+k$.
Hence,
\begin{equation}\label{ser-WL}
\sum\limits_{j=1}^{t-1} |W\cup L|^{j}(n-|W\cup L|)^{t-1-j}= O(\varepsilon_2 n^{t-1}).
\end{equation}
Combining  (\ref{ev-xv0})-(\ref{ser-WL}), we obtain
$$\rho_{t}(G)x_{v_{0}}^{t-1}\leq \dbinom{r-1}{t-1}\left(\frac{d_{G}(v_{0})}{r-1}\right)^{t-1}x_{v_{0}}^{t-1}+O(n^{t-2})
+O(\varepsilon_2 n^{t-1}).
$$

If $v_{0}\in L$, then by definition, $d_{G}(v_{0}) \leq (1 - \tfrac{1}{r} - \varepsilon_{1})n$.
Therefore,
$$\dbinom{r-1}{t-1}\left(\frac{d_{G}(v_{0})}{r-1}\right)^{t-1}=
\dbinom{r-1}{t-1}\left(\frac{n}{r}\right)^{t-1}-\Theta(\varepsilon_{1}n^{t-1})+O(\varepsilon_{1}^{2}n^{t-1}).$$
Since $x_{v_{0}}^{t-1}=\Theta(1)$, we have
\begin{equation}\label{rho-contr}
\rho_{t}(G)x_{v_{0}}^{t-1}
\leq \dbinom{r-1}{t-1}\left(\frac{n}{r}\right)^{t-1}x_{v_{0}}^{t-1}-\left[\,\Theta(\varepsilon_{1})-O(\varepsilon_1^2)-O(\varepsilon_{2})\right] n^{t-1}+O(n^{t-2}).
\end{equation}
Since $\varepsilon_{2} \ll \varepsilon_{1}$, for sufficiently large $n$, we have
$$\rho_{t}(G)< \dbinom{r-1}{t-1}\left(\frac{n}{r}\right)^{t-1}.$$
This contradicts the lower bound given by Lemma \ref{lowerbound}.
Hence, $v_{0}\notin L$.
\end{proof}

With Lemmas \ref{v-neigh}, \ref{v-indeg}, and \ref{2nd-ev} in hand, we now prove that $L = \emptyset$.

\begin{lem}\label{LO}
 $L = \emptyset$.
\end{lem}
\begin{proof}

Suppose, to the contrary, that there is a vertex $u\in L$.
Let $x_{u_0}=\max\{x_{v}:v\in V(G)\}=1$ and $x_{v_0}=\max\{x_{v} : v \in V(G)\setminus W\}$.
By Lemma \ref{2nd-ev}, we may assume without loss of generality that $v_{0}\in V_{1}\setminus (W\cup L)$.

By Lemma \ref{Ni-Turan} (d), for each $i\in[r]$, there exists an independent set $I_{i}\subseteq V_{i}\setminus (W\cup L)$ such that
$|I_{i}|\geq |V_{i}\setminus (W\cup L)|-2(k-1)$.
Construct a graph $G'$ from $G$ by deleting all edges incident to $u$ and then joining $u$ to every vertex in $\bigcup_{i=2}^{r}I_{i}$.
Clearly, $G'$ is $kK_{r+1}$-free.

We show that $\rho_{t}(G')> \rho_{t}(G)$.
Partition the family $C_t(v_0)$ into four classes:
\begin{align*}
\mathcal{S}_1 &= \left\{ e \in C_{t}(v_0): (e \setminus \{v_0\}) \cap (W \cup L) \neq \emptyset \right\}, \\
 \mathcal{S}_2 &= \left\{ e \in C_{t}(v_0):(e \setminus \{v_0\}) \cap (W \cup L) = \emptyset, (e \setminus \{v_0\}) \cap V_1 \neq \emptyset \right\}, \\
\mathcal{S}_3 &= \left\{ e \in C_{t}(v_0):
  (e \setminus \{v_0\}) \subseteq V(G) \setminus (V_{1} \cup W \cup L), (e \setminus \{v_0\}) \cap (\textstyle\bigcup_{i=2}^{r} (V_{i} \setminus I_{i})) \neq \emptyset \right\}, \\
 \mathcal{S}_4 &= \left\{ e \in C_{t}(v_0): (e \setminus \{v_0\}) \subseteq \textstyle\bigcup_{i=2}^{r} I_{i} \right\}.
\end{align*}
Write $S_i = \sum_{e  \in \mathcal{S}_i}x^{e \setminus \{v_0\}}$ for $i=1,2,3,4$.
Then
\begin{equation}\label{ev-xv0-dec}
\rho_{t}(G)x_{v_{0}}^{t-1}=\sum_{e \in C_{t}(v_0)}x^{e \setminus \{v_0\}}=S_1+S_2+S_3+S_4,
\end{equation}

We estimate each term.
Since $ |W\cup L|\leq2\varepsilon_{2}n+k$,
\begin{equation}\label{S1-cont}
S_1 = \sum_{e  \in \mathcal{S}_1 }x^{e \setminus \{v_0\}}
\leq \sum_{j=1}^{t-1} |W\cup L|^{j} n^{t-1-j}= O(\varepsilon_2 n^{t-1}).
\end{equation}
By Lemma \ref{v-indeg}, $d_{V_{1}\setminus (W\cup L)}(v)\leq (k-1)(r+1)$ for any $v\in V_{1}\setminus (W\cup L)$.
Thus,
$$
S_2 = \sum_{e \in \mathcal{S}_2}x^{e \setminus \{v_0\}}
\leq\sum_{j=1}^{t-1} d_{V_{1}\setminus (W\cup L)}(v_{0})^{j}
n^{t-1-j}= O(n^{t-2}).
$$
For each $i\in[r]$, Lemma \ref{Ni-Turan} (d) gives $|V_{i}\setminus (W\cup L)|-|I_{i}|\leq 2(k-1)$.
So,
$$
S_3 = \sum_{e  \in \mathcal{S}_3}x^{e \setminus \{v_0\}}
\leq \sum_{j=1}^{t-1} (2(k-1))^{j} n^{t-1-j}= O(n^{t-2}).
$$
Therefore, from \eqref{ev-xv0-dec},
$$S_4 \ge \rho_{t}(G)x_{v_{0}}^{t-1}-O(\varepsilon_2 n^{t-1}) - O(n^{t-2}).$$
In $G'$, $u$ is adjacent to all vertices of $\bigcup_{i=2}^{r} I_i$ in $G'$,
so, combining Lemma \ref{lowerbound},
\begin{equation}\label{G'-S4}
\begin{split}
\sum_{e \in \binom{\bigcup_{i=2}^{r} I_{i}}{t-1}, e \cup \{u\}\in C_{t}(G^{\prime})}x^e
& \geq S_4 \ge \rho_{t}(G)x_{v_{0}}^{t-1}-O(\varepsilon_2 n^{t-1}) - O(n^{t-2})\\
& \geq \binom{r-1}{t-1}\left(\frac{n}{r}\right)^{t-1}x_{v_{0}}^{t-1}-O(\varepsilon_2 n^{t-1})-O(n^{t-2}).
\end{split}
\end{equation}

On the other hand, arguing as in  (\ref{rho-contr}), we obtain
\begin{equation}\label{rho-xu-2}
\rho_{t}(G)x_{u}^{t-1}
\leq\dbinom{r-1}{t-1}\left(\frac{n}{r}\right)^{t-1}x_{v_{0}}^{t-1}-\left[\,\Theta(\varepsilon_{1})-O(\varepsilon_1^2)-O(\varepsilon_{2})\right] n^{t-1}+
O(n^{t-2}).
\end{equation}
Since
$$\rho_{t}(G)x_{u}^{t-1}=\sum_{e \in \binom{N_{G}(u)}{t-1},e \cup \{u\}\in C_{t}(G)}x^e.
$$
combining  (\ref{G'-S4}) and  (\ref{rho-xu-2}), we obtain
 $$
 \begin{aligned}
\|x\|_t^t\left(\rho_{t}(G^{\prime})-\rho_{t}(G)\right)&\geq x^\top\A_t(G^{\prime})x^{t-1}-x^\top\A_t(G)x^{t-1}\\
 &=tx_{u}\bigg(\sum_{e \in \binom{\bigcup_{i=2}^{r} I_{i}}{t-1},e \cup \{u\}\in C_{t}(G^{\prime})}x^e
-\sum_{e \in \binom{N_{G}(u)}{t-1}, e \cup \{u\}\in C_{t}(G)}x^e\bigg) \\
&\geq  tx_{u}\left(\left( \Theta(\varepsilon_{1}) - O(\varepsilon_{1}^2) - O(\varepsilon_{2}) \right)n^{t-1} - O(n^{t-2})\right) \\
&>0,
\end{aligned}
$$
where the last inequality holds since $\varepsilon_{2} \ll \varepsilon_{1}$ and $n$ is sufficiently large.
Hence, $\rho_{t}(G^{\prime})>\rho_{t}(G)$, contradicting the extremality of $G$.
Therefore, $L$ must be empty.
\end{proof}

For a general vertex $v$ of $G$, we now derive a uniform lower bound on $x_v$, independent of $n$.

\begin{lem}\label{ev}
For every $v\in V(G)$, $$x_{v}^{t-1}\geq \frac{1}{(t-1)!}\left(1-\frac{t-1/3}{r}\right)^{t-1}.$$
\end{lem}

\begin{proof}
Suppose, to the contrary, that there exists a vertex $u\in V (G)$ such that
$$x_{u}^{t-1}< \frac{1}{(t-1)!}\left(1-\frac{t-1/3}{r}\right)^{t-1}.$$
Proceeding as in the proof of Lemma \ref{LO}, for each $i \in [r]$, let $I_i \subseteq V_i \setminus (W \cup L)$ be an independent set with $|I_i| \ge |V_i \setminus (W \cup L)|-2(k-1)$.
Define a new graph $G'$ by deleting all edges incident to $u$ and then joining $u$ to every vertex in $\bigcup_{i=2}^{r}I_{i}$.
Clearly, $G'$ is $kK_{r+1}$-free.
Let $x_{v_0}=\max\{x_{v} : v \in V(G)\setminus W\}$.
Since $L=\emptyset$ by Lemma \ref{LO}, Lemma \ref{Ni-Turan}(e) yields $|W\cup L|=|W| \le k-1$.
Thus, all error terms involving $W$ improve; the inequality in \eqref{S1-cont} is $S_1=O(n^{t-2})$, and the inequality (\ref{G'-S4}) can be improved to
\begin{equation}\label{imp}
\sum_{e \in \binom{\bigcup_{i=2}^{r} I_{i}}{t-1}, ~ e \cup \{u\}\in C_{t}(G^{\prime})}x^e \geq
\rho_{t}(G)x_{v_{0}}^{t-1}-O(n^{t-2}).
\end{equation}

By Lemma \ref{2nd-ev} we have
$x_{v_{0}}^{t-1}> \frac{1}{(t-1)!}\left(1-\frac{t-1/2}{r}\right)^{t-1}$, and hence
$x_{v_{0}}^{t-1} - x_u^{t-1} = \Theta(1)$.
Using the eigenvector identity and arguing as before, and combining \eqref{imp}, we have
$$
\begin{aligned}
\|x\|_t^t\left(\rho_{t}(G^{\prime})-\rho_{t}(G)\right)&\geq x^\top\A_t(G^{\prime})x^{t-1}-x^\top\A_t(G)x^{t-1}\\
&=tx_{u}\bigg(\sum_{e \in \binom{\bigcup_{i=2}^{r} I_{i}}{t-1}, e \cup \{u\}\in C_{t}(G^{\prime})}x^e
-\sum_{e \in \binom{N_{G}(u)}{t-1}, e \cup \{u\}\in C_{t}(G)}x^e\bigg)\\
&\geq tx_{u}\bigg(\sum_{e \in \binom{\bigcup_{i=2}^{r} I_{i}}{t-1},e \cup \{u\}\in C_{t}(G^{\prime})}x^e
-\rho_{t}(G)x_{u}^{t-1}\bigg)\\
&\geq tx_{u}\left(\rho_{t}(G)(x_{v_{0}}^{t-1}-x_{u}^{t-1}) - O(n^{t-2})\right).
\end{aligned}
$$
By Lemma \ref{lowerbound}, $\rho_t(G) \ge \Theta(n^{t-1})$, and thus the right-hand side is positive for sufficiently large $n$, implying that
$\rho_t(G') > \rho_t(G)$, a contradiction.
\end{proof}

The next lemma shows that $W$ has size $k-1$ and $G-W$ is $r$-partite.
Recall that $\mu(G)$ denotes the matching number of a graph $G$.

\begin{lem}\label{W}
$|W| = k-1$, and $V_{i}\setminus W$ is an independent set for each $i \in [r]$.
\end{lem}

\begin{proof}
Since $L=\emptyset $ by Lemma \ref{LO}, Lemma \ref{Ni-Turan}(e) yields $|W| \leq k-1$.
We first prove that
\begin{equation}\label{match}
\mu(\cup_{i=1}^{r}G[V_{i} \setminus W]) \leq k-1-|W|.
\end{equation}
Suppose, to the contrary, that there exists a matching $M$ of size $k-|W|$ in $\bigcup_{i=1}^{r}G[V_{i} \setminus W]$.
Since $L$ is empty, we have $\delta(G) \geq \left(1 - \frac{1}{r} - \varepsilon_1\right)n$.
Using the same neighborhood-expansion argument as in the proof of \eqref{S-neigh}, we process the edges of $M$ one by one.
For each edge $e = xy \in M$, suppose that $e \subseteq V_i \setminus W$ for some $i$.
We can find a copy of $K_{r+1}$ in the current graph that contains $x,y$ and exactly one vertex from each $V_{j} \setminus W$ (with $j \neq i$), and then delete the vertices of this $K_{r+1}$ from the graph.
We repeat this procedure (each time working on the remaining graph) until we have constructed $k-|W|$ vertex-disjoint copies of $K_{r+1}$, say $C_{1},\dots,C_{k-|W|}$, all disjoint from $W$.

Let $G'$ be the remaining graph, i.e. the graph from $G$ by deleting all vertices of the above cliques.
We now construct $|W|$ additional vertex-disjoint copies of $K_{r+1}$ in $G'$, each containing a distinct vertex of $W$.
Let $V(G') = V_{1}' \cup \cdots \cup V_{r}'$ according to the partition of $G$.
Take any vertex $w \in W$, say $w \in V_{1}'$.
Since $d_{V_{1}'}(w) \geq 2\theta n-(k-|W|)(r+1)$ and $|W| \le k-1$,
the vertex $w$ has a neighbor $v$ in $V_{1}' \setminus W$.
Moreover,
$$\delta(G') \geq \left(1 - \frac{1}{r} - \varepsilon_1\right)n-(k-|W|)(r+1).$$
Applying the same argument again, we can find a $K_{r+1}$ in $G'$ containing $w$ and $v$.
Delete the vertices of this $K_{r+1}$ and repeat the procedure for all vertices of $W$.
Then, we obtain $|W|$ additional vertex-disjoint copies of $K_{r+1}$, each of which contains a vertex from $W$ and is vertex-disjoint from the cliques $C_{1},\ldots,C_{k-|W|}$.
Altogether, $G$ has $k$ vertex-disjoint copies of $K_{r+1}$, a contradiction.
Thus \eqref{match} holds.

Next, we pay attention to show that $|W|= k-1$.
Suppose that $|W|< k-1$.
Let
$$\Delta:=\Delta(\cup_{i=1}^{r}G[V_{i} \setminus W]), ~ m:=\mu(\cup_{i=1}^{r}G[V_{i} \setminus W]).$$
By Lemma \ref{v-indeg}, $\Delta \leq (k-1)(r+1)$, and by \eqref{match}, $m \leq k-1-|W|$.
Thus, by Lemma \ref{chv-md},
\begin{eqnarray}\label{edge-ViW}
e(\cup_{i=1}^{r}G[V_{i} \setminus W])\leq f(m,\Delta)\leq \Delta m + m\leq (kr+k-r)(k-1-|W|).
\end{eqnarray}

We construct a new graph $G''$ with the same vertex set as $G$ as follows.
Choose a set $S \subseteq V_{1} \setminus W$ with $|S| = k-1-|W|$.
and let $E^{*}$ be the edge set of the complete bipartite graph between $S$ and $V_{1} \setminus (W\cup S)$.
The graph $G''$ is obtained from $G$ by removing all edges inside $V_{i} \setminus W$ for $ i \in [r]$ and adding all edges of $E^*$.
Since any $K_{r+1}$ in $G''$ must contain one vertex from $W\cup S$, and $|W\cup S|=k-1$, it follows that $G''$ contains at most $k-1$ vertex-disjoint copies of $K_{r+1}$. Thus $G''$ is $kK_{r+1}$-free.

We now prove that $\rho_{t}(G'')> \rho_{t}(G)$.
Let $C_{t}^{\mathrm{in}}(G)$ be the set of $t$-cliques in $G$ containing at least one edge from $\bigcup_{i=1}^{r} G[V_i \setminus W]$.
Note that for any edge $uv$ in $\bigcup_{i=1}^{r} G[V_{i} \setminus W]$,
 the number of $t$-cliques containing $uv$ is at most $n^{t-2}$.
Since $e( \bigcup_{i=1}^{r} G[V_{i} \setminus W] ) = O(1)$ by \eqref{edge-ViW}, we obtain
$$|C_{t}^{\mathrm{in}}(G)|
\leq e\left( \cup_{i=1}^{r} G[V_{i} \setminus W] \right) \cdot n^{t-2} = O(n^{t-2}).$$

On the other hand, let $C_{t}^{*}(G'')$ be the set of $t$-cliques in $G''$ which contain at least one edge from $E^{*}$.
Since $L = \emptyset$, we have $\delta(G) \geq \left(1-\frac{1}{r}-\varepsilon_{1}\right)n.$
Combining (\ref{edge-ViW}), we have
$$\delta(G'') \geq \delta(G) - (kr+k-r)(k-1-|W|) \geq \left(1-\frac{1}{r}-\varepsilon_{1}\right)n - O(1).$$
Thus, by the same argument as in the proof of inequality (\ref{S-neigh}),  for every edge $uv \in E^{*}$
the induced subgraph $G[N_{G''}(u) \cap N_{G''}(v)]$ contains a complete $(r-1)$-partite subgraph with each parts of size at least $\frac{n}{r+1}$.
Consequently, each such edge $uv$ lies in at least
$(\frac{n}{r+1})^{t-2}$ distinct $t$-cliques.
Since $|E^{*}| = |S||(V_{1} \setminus (W\cup S)| \geq \frac{n}{r+1}$, we obtain
$$|C_{t}^{*}(G'')|
\geq |E^{*}| \cdot \left(\frac{n}{r+1}\right)^{t-2}
= \Omega(n^{t-1}).$$

Let $x$ be a positive eigenvector of $\A_t(G)$ corresponding to the spectral radius $\rho_t(G)$.
We normalize the eigenvector so that $\max\{x_{v}: v \in V(G)\} = 1$.
By Lemma \ref{ev}, there exists a constant $c_0>0$ such that $x_{v} \geq c_{0}$ for all $v \in V(G)$.
Thus,
$$\sum_{e \in C_{t}^{*}(G'')} x^e
\geq c_{0}^{\,t} |C_{t}^{*}(G'')| = \Omega(n^{t-1}), ~ \sum_{e \in C_{t}^{\mathrm{in}}(G)} x^e \leq |C_{t}^{\mathrm{in}}(G)| = O(n^{t-2}).$$
Therefore, for sufficiently large $n$,
$$
\begin{aligned}
\|x\|_t^t\left(\rho_{t}(G')-\rho_{t}(G)\right)
&\geq x^{T}\A_t(G'')x^{t-1}-x^{T}\A_t(G)x^{t-1}\\
&=t\ \Bigg(\sum_{e \in C_{t}^{*}(G'')}x^e
-\sum_{e \in C_{t}^{\mathrm{in}}(G)}x^e\Bigg)\\
& \ge t \left( \Omega(n^{t-1}) - O(n^{t-2}) \right)\\
& > 0,
\end{aligned}
$$
which contradicts the extremality on $G$.
Thus, $|W| = k-1$.
Finally, \eqref{match} implies that $\mu(\cup_{i=1}^{r}G[V_{i} \setminus W]) =0$, and hence each $V_{i}\setminus W$ is an independent set.
\end{proof}

The following lemma, together with Lemma \ref{W}, shows that $G$ is a join of a complete graph $K_{k-1}$ (on the vertex set $W$) and an $r$-partite graph.

\begin{lem}\label{W-dom}
For every vertex $w\in W$, $d_{G}(w) = n-1$.
\end{lem}

\begin{proof}
Suppose, to the contrary, that there exists $w\in W$ such that $d_{G}(w) < n-1$.
Then there exists a vertex $v\in V (G)$ such that $wv\notin E(G)$.
Let $G'$ be obtained from $G$ by adding the edge $wv$.
Since $|W|=k-1$ and each $G[V_i \setminus W]$ is an independent set for $i \in [r]$, it follows that $G'$ is still $kK_{r+1}$-free.
Moreover, by Lemma \ref{t-clq-conn}, $G$ is $t$-clique connected, and hence so is $G'$.
In addition, by the same argument as in Lemma \ref{v-indeg} (or Lemma \ref{W}), the graph $G'$ contains a $t$-clique using the edge $wv$.
Therefore, by Lemma \ref{Khan-rho}, we have $\rho_{t}(G')> \rho_{t}(G)$, contradicting the extremality of $G$.
Hence $d_G(w) = n-1$ for all $w \in W$.
\end{proof}

With the above preparations, we are ready to prove the main theorem.

\begin{proof}[\bf Proof of Theorem \ref{Main}]
We divide the proof into two claims.

\begin{claim}\label{C1}
$G \cong K_{k-1} \vee K_{r}(s_{1}, s_{2}, \ldots, s_{r})$,
where $K_{r}(s_{1}, s_{2}, \ldots, s_{r})$ denotes the complete $r$-partite graph with part sizes $s_{1}, \ldots, s_{r}$.
\end{claim}

By Lemmas \ref{W} and \ref{W-dom}, we have $G = K_{k-1} \vee H$, where $W=V(K_{k-1})$ and $H$ is an $r$-partite graph with parts $V_1\setminus W, \ldots, V_r\setminus W$.
Let $s_i=|V_i \setminus W|$ for $i \in [r]$.
Suppose that $H \not\cong K_{r}(s_{1}, s_{2}, \ldots, s_{r})$.
Then there exist vertices $u \in V_{i} \setminus W$ and $v \in V_{j} \setminus W$ with $i \neq j$ such that $uv \notin E(G)$.
Let $G'$ be the graph obtained from $G$ by adding the edge $uv$.
Following a similar argument as in the proof of Lemma \ref{W-dom}, $G'$ remains $kK_{r+1}$-free satisfying $\rho_t(G') > \rho_t(G)$, a contradiction.
Hence $H \cong K_{r}(s_{1}, s_{2}, \ldots, s_{r})$.

Without loss of generality, assume $s_1 \ge s_2 \ge \cdots \ge s_r$.

\begin{claim}\label{C2}
If there exist $i$ and $j$ with $1\leq i < j\leq r$ such that $s_{i}-s_{j}\geq 2$,
then
$$\rho_{t}(K_{k-1} \vee K_{r}(s_{1},\ldots, s_{i}-1, \ldots, s_{j}+1,\ldots,s_{r}))
> \rho_{t}(K_{k-1} \vee K_{r}(s_{1},\ldots, s_{i}, \ldots, s_{j},\ldots,s_{r})).$$
\end{claim}

Let $G_{1} = K_{k-1} \vee K_{r}(s_{1},\ldots, s_{r})$.
Partition $V(G_{1})$ as $W\cup V_{1}\cup\cdots\cup V_{r}$, where
$W=V(K_{k-1})$, and $V_{i}$ is the part of $K_{r}(s_{1},\ldots, s_{r})$ of size $s_i$ for $i \in [r]$.
Clearly, $G_{1}$ is $t$-clique connected.
Let $x$ be a positive eigenvector of $\A_t(G_1)$ corresponding to $\rho_{t}(G_{1})$.
By symmetry, we may assume
$x|_{W}=\alpha_0$ and $x|_{V_i}=\alpha_i$ for $i \in [r]$.
Note that for any two distinct pairs $(u_1, v_1), (u_2, v_2) \in V_i \times V_j$,
 $N(u_1) \cap N(v_1) = N(u_2) \cap N(v_2)$.
Thus, for any $u\in V_{i}$ and $v\in V_{j}$,  we can define
$$\beta_{1}= \sum_{\substack{\{u,v,i_{3},\ldots,i_{t}\}\in C_{t}(G_{1})}} x_{i_{3}}\cdots x_{i_{t}},$$
and
$$\beta_{2}= \sum_{\substack{\{u,i_{2},\ldots,i_{t}\}\in C_{t}(G_{1}) \\ \{i_{2},\ldots,i_{t}\}\cap V_{j}=\emptyset}} x_{i_{2}}\cdots x_{i_{t}} = \sum_{\substack{\{v,i_{2},\ldots,i_{t}\}\in C_{t}(G_{1}) \\ \{i_{2},\ldots,i_{t}\}\cap V_{i}=\emptyset}} x_{i_{2}}\cdots x_{i_{t}}.$$
Furthermore, by the eigenvector equation at the vertex $u$,
\begin{equation}\label{ev-ai}
\rho_{t}(G_{1})\alpha_i^{t-1}=\sum_{\{u,i_{2},\ldots,i_{t}\}\in C_{t}(G_{1})}x_{i_{2}}\cdots x_{i_{t}}=
s_{j}\alpha_j \beta_{1}+\beta_{2},
\end{equation}
and by the eigenvector equation at the vertex $v$,
\begin{equation}\label{ev-aj}
\rho_{t}(G_{1})\alpha_j^{t-1}=\sum_{\{v,i_{2},\ldots,i_{t}\}\in C_{t}(G_{1})}x_{i_{2}}\cdots x_{i_{t}}=
s_{i}\alpha_i\beta_{1}+\beta_{2}.
\end{equation}
Combining  (\ref{ev-ai}) and  (\ref{ev-aj}), we obtain
$$\rho_{t}(G_{1})(\alpha_i^{t-1}-\alpha_j^{t-1})=(s_{j}\alpha_j-s_{i}\alpha_i)\beta_1.$$
So, if $\alpha_i\geq \alpha_j$, then $s_{j} \geq s_{i}$, contradicting $s_{i} \geq s_{j}+2$.
Hence $\alpha_i < \alpha_j$.

Let $G_{2} = K_{k-1} \vee K_{r}(s_{1},\ldots, s_{i}-1, \ldots, s_{j}+1,\ldots,s_{r})$,
and set $E_{1}=E(G_{1})\setminus E(G_{2})$ and $E_{2}=E(G_{2})\setminus E(G_{1})$.
Then, we have
\begin{equation}\label{rho-G1G2}
\begin{aligned}
\|x\|_t^t\left(\rho_{t}(G_{2})-\rho_{t}(G_{1})\right)
&\geq x^\top\A_t(G_{2})x^{t-1}-x^\top\A_t(G_{1})x^{t-1}\\
 &=t\Bigg(\sum_{\{i_{1},\ldots,i_{t}\}\in C_{t}(G_{2})}x_{i_{1}}\cdots x_{i_{t}}-\sum_{\{i_{1},\ldots,i_{t}\}\in C_{t}(G_{1})}x_{i_{1}}\cdots x_{i_{t}}\Bigg)\\
&= t\Bigg(\sum_{\{i_{1},i_{2}\}\in E_{2}}x_{i_{1}}x_{i_{2}}
 -\sum_{\{i_{1},i_{2}\}\in E_{1}}x_{i_{1}}x_{i_{2}}\Bigg)\beta_{1}\\
 &= t\left((s_{i}-1)\alpha_i^{2}-s_{j}\alpha_i\alpha_j\right)\beta_{1}.
\end{aligned}
\end{equation}
It remains to show $(s_{i}-1)\alpha_i - s_{j}\alpha_j > 0$.

By \eqref{ev-ai} and \eqref{ev-aj}, we have
\begin{equation}\label{rho_uv}
\rho_{t}(G_{1})\alpha_i^{t-1}+s_{i}\alpha_i \beta_{1}=\rho_{t}(G_{1})\alpha_j^{t-1}+s_{j}\alpha_j \beta_{1}>0.
\end{equation}
Since $\alpha_i< \alpha_j$, by \eqref{ev-ai}, we can derive
\begin{equation}\label{rho_ab}
\rho_{t}(G_{1})\alpha_j^{t-2}>\rho_{t}(G_{1})\alpha_i^{t-2}=s_{j}\beta_{1}\frac{\alpha_j}{\alpha_i}+\frac{\beta_{2}}{\alpha_i}>s_{j} \beta_{1}.
\end{equation}
Consequently, by \eqref{rho_uv}, \eqref{rho_ab} and the fact $s_i-1 \ge s_j +1$, we have
\begin{align*}
&\big((s_{i} - 1)\alpha_i - s_{j} \alpha_j\big)\big(\rho_{t}(G_{1})\alpha_j^{t-1} + s_{j} \alpha_j \beta_1\big) \\
&= (s_{i} - 1)\alpha_i\big(\rho_{t}(G_{1})\alpha_j^{t-1} + s_{j} \alpha_j \beta_1\big) - s_{j} \alpha_j\big(\rho_{t}(G_{1})\alpha_i^{t-1} + s_{i} \alpha_i \beta_1\big) \\
&= \alpha_i \alpha_j\left((s_{i} - 1)\big(\rho_{t}(G_{1})\alpha_j^{t-2} + s_{j} \beta_1\big) - s_{j}\big(\rho_{t}(G_{1})\alpha_i^{t-2} + s_{i} \beta_1\big)\right) \\
&= \alpha_i \alpha_j\left(\rho_{t}(G_{1})\alpha_j^{t-2}(s_{i} - 1) - s_{j}\big(\rho_{t}(G_{1})\alpha_i^{t-2} + \beta_1\big)\right) \\
&\geq \alpha_i \alpha_j\left(\rho_{t}(G_{1})\alpha_j^{t-2}(s_{j} + 1) - s_{j}\big(\rho_{t}(G_{1})\alpha_i^{t-2} + \beta_1\big)\right) \\
&= \alpha_i \alpha_j\left(s_{j}\rho_{t}(G_{1})(\alpha_j^{t-2} - \alpha_i^{t-2}) + \rho_{t}(G_{1})\alpha_j^{t-2} - s_{j} \beta_1\right) \\
&> 0.
\end{align*}
Since $(\rho_{t}(G_{1})\alpha_j^{t-1} + s_{j} \alpha_j \beta_1)>0$,  we have $(s_{i}-1)\alpha_i - s_{j}\alpha_j > 0$.
So, $\rho_{t}(G_{2}) >\rho_{t}(G_{1})$ by \eqref{rho-G1G2}, as desired.

Finally,  Claims \ref{C1} and Claim \ref{C2} together imply that
if $G$ is a $kK_{r+1}$-free graph on $n$ vertices with maximum $t$-clique spectral radius, then
$G \cong K_{k-1} \vee T_{r}(n-k+1)$.
This completes the proof of Theorem \ref{Main}.
\end{proof}


\begin{thebibliography}{99}
\setlength{\itemsep}{0pt}
\small

\bibitem{Chen}
M. Chen, X. Zhang, Some new results and problems in spectral extremal graph theory, \emph{J. Anhui Univ. Nat. Sci.}, 42 (2018), 12--25.

\bibitem{Chv}
V. Chv\'atal, D. Hanson, Degrees and matchings. \emph{J. Combin. Theory Ser. B}, 20 (2) (1976), 128--138.

\bibitem{Cio2}
S. Cioab\u{a}, D. N. Desai, M. Tait, The spectral even cycle problem, \emph{Combin. Theory}, 4 (1) (2024), 10.

\bibitem{Cio}
S. Cioab\u{a}, L. Feng, M. Tait, X. Zhang, The maximum spectral radius of graphs without friendship subgraphs, \emph{Electron. J. Combin.}, 27 (4) (2020), P4.22.

\bibitem{Cooper}
J. Cooper, A. Dutle, Spectra of uniform hypergraphs, \emph{Linear Algebra Appl.}, 436 (2012), 3268--3292.

\bibitem{Fang}
L. Fang, H. Lin, Spectral extremal results on edge blow-up of graphs, \emph{Discrete Math.}, 349 (2) (2026), 114835.

\bibitem{Friedland}
S. Friedland, S. Gaubert, L. Han, Perron-Frobenius theorem for nonnegative multilinear forms and extensions, \emph{Linear Algebra Appl.}, 438 (2013), 738--749.

\bibitem{Gernber2}
D. Gerbner, Some exact results for non-degenerate generalized Tur\'an problems,  \emph{Electron. J. Combin.}, 30(4)(2023), P4.39.

\bibitem{Gernber}
D. Gernber, Generalized Tur\'an results for disjoint cliques, \emph{Discrete Math.}, 347(7)(2024), 114024.

\bibitem{Gernber3}
D. Gerbner, C. Palmer, Survey of generalized Tur\'an problems - counting subgraphs, \emph{Electron. J. Combin.}, 2026, Dynamic Surveys, DS27.

\bibitem{Khan}
M. Khan, Y.-Z. Fan, On the spectral radius of a class of non-odd-bipartite even uniform hypergraphs,
\emph{Linear Algebra Appl.}, 480(2015), 93--106.

\bibitem{Li}
Y. Li, W. Liu, L. Feng, A survey on spectral conditions for some extremal graph problems, \emph{Advances in Math. (China)}, 51 (2) (2022), 193--258.

\bibitem{Lim}
L. Lim, Singular values and eigenvalues of tensors: a variational approach, \emph{1st IEEE International Workshop on Computational
Advances in Multi-Sensor Adaptive Processing}, 2005, 129--132.

\bibitem{Liu3}
X. Liu, Spectral generalized Tur\'an problems, arXiv:2507.21689.

\bibitem{Liu2}
C. Liu, C. Bu,  On a generalization of the spectral Mantel's theorem,  \emph{J. Comb. Optim.}, 46(2)(2023), 14.

\bibitem{Liu}
C. Liu, C. Bu, A tensor's spectral bound on the clique number, \emph{Discrete Math.}, 349 (2026), 114694.

\bibitem{Liu.R}
R. Liu, L. Miao, Spectral Tur\'an problem of non-bipartite graphs: Forbidden books, \emph{European J. Combin.}, 126 (2025), 104136.

\bibitem{Liu.L}
L. Liu, B. Ning, A spectral analogue of Ore's problem on Tur\'an theorem, \emph{Linear Algebra Appl.}, 735 (3) (2026), 112--122.

\bibitem{Moon}
J. Moon, On independent complete subgraphs in a graph, \emph{Canad. J. Math.}, 20(1968), 95--102.

\bibitem{Ni}
Z. Ni, J. Wang, L. Kang, Spectral extremal graphs for disjoint cliques, \emph{Electron. J. Combin.}, 30(1) (2023), P1.20.

\bibitem{Nikiforov2}
V. Nikiforov, Bounds on graph eigenvalues II, \emph{Linear Algebra Appl.} 427 (2007), 183--189.

\bibitem{Nikiforov}
V. Nikiforov, The spectral radius of graphs without paths and cycles of specified length, \emph{Linear Algebra Appl.},  432 (2010), 2243--2256.

\bibitem{Qi1}
L. Qi, Eigenvalues of a real supersymmetric tensor, \emph{J. Symb. Comput.}, 40 (2005), 1302--1324.

\bibitem{Qi2}
L. Qi, Symmetric nonnegative tensors and copositive tensors, \emph{Linear Algebra Appl.}, 439 (2013), 228--238.

\bibitem{Re}
A. Rehman, S. Pirzada, The spectral Tur\'an problem about graphs of given size with forbidden subgraphs, \emph{AKCE Int. J. Graphs Comb.}, 22 (1) (2025), 91--93.

\bibitem{Shao}
J. Shao, H. Shan, L. Zhang, On some properties of the determinants of tensors, \emph{Linear Algebra Appl.}, 439 (2013), 3057--3069.

\bibitem{Simonovits}
M. Simonovits, A method for solving extremal problems in graph theory, stability problems, \emph{Theory of Graphs} (Proc. Colloq., Tihany, 1996), Academic Press, New York, 1968, 279--319.

\bibitem{Turan}
P. Tur\'an, On an extremal problem in graph theory, \emph{Mat. Lapok}, 48 (1941) 436--452.

\bibitem{Peng3}
Z. Yan, B. Yang, Y. Peng, A spectral generalized Alon-Frankl theorem, \emph{Discrete Math.}, 349 (2) (2026), 114785.

\bibitem{Yang2}
Y. Yang, Q. Yang, Further results for Perron-Frobenius theorem for nonnegative tensors, \emph{SIAM J. Matrix Anal. Appl.}, 31 (2010), 2517--2530.

\bibitem{Yang1}
Y. Yang, Q. Yang, On some properties of nonnegative weakly irreducible tensors, arXiv: 1111.0713v1.

\bibitem{Peng}
L. Yu, Y. Peng, A spectral version of the theorem of Zykov and Erd\H{o}s, \emph{Electron. J. Combin.}, 32 (4) (2025), 2517--2530.

\bibitem{Peng2}
L. Yu, Y. Peng, A spectral generalized Erd\H{o}s-Gallai theorem, \emph{Discrete Math.}, 349 (1) (2026), 114672.

\bibitem{Zhai}
M. Zhai, M. Liu, Extremal problems on planar graphs without $k$ edge-disjoint cycles, \emph{Adv. Appl. Math.}, 157 (2024), 102701.

\bibitem{Zykov}
A.A. Zykov, On some properties of linear complexes, \emph{Matematicheskii Sbornik}, 66(2) (1949), 163--188.

\end{thebibliography}
\end{document}